\documentclass[a4paper]{amsart}
\usepackage{graphicx}
\usepackage{amssymb}
\usepackage{amsmath}
\usepackage{amsthm}
\usepackage{amscd}
\usepackage[all,2cell]{xy}

\UseAllTwocells \SilentMatrices
\newtheorem{thm}{Theorem}[section]

\newtheorem{lem}[thm]{Lemma}
\newtheorem{exm}[thm]{Example}

\theoremstyle{definition}

\theoremstyle{remark}
\newtheorem{rem}[thm]{\bf Remark}
\numberwithin{equation}{section}

\begin{document}
\title[Wakamatsu's equivalence revisited]{Wakamatsu's equivalence revisited}
\author[Xiao-Wu Chen, Jiaqun Wei] {Xiao-Wu Chen, Jiaqun Wei$^*$}

\thanks{$^*$ The corresponding author}
\thanks{This work is supported by National Natural Science Foundation of China (Nos. 11522113, 11671245 and 11371196), and a project funded by the Priority Academic Program Development of Jiangsu Higher Education Institutions.}
\date{\today}

\thanks{E-mail: xwchen$\symbol{64}$mail.ustc.edu.cn; weijiaqun@njnu.edu.cn}

\maketitle

\dedicatory{}%
\commby{}%

\begin{abstract}
For a certain Wakamatsu-tilting bimodule over two artin algebras $A$ and $B$, Wakamatsu constructed an explicit equivalence between the stable module categories over  the trivial extension algebra of $A$ and that of $B$. We prove that Wakamatsu's functor is a triangle functor, thus a triangle equivalence.
\end{abstract}

\section{Introduction}

Let $A$ and $B$ be two artin algebras. Denote by $T(A)$ and $T(B)$ their trivial extension algebras.  For a certain Wakamatsu-tilting bimodule $_AT_B$, Wakamatsu constructed  in \cite{Wak2} an explicit equivalence between the stable module categories  $T(A)\mbox{-\underline{\rm mod}}$  and $T(B)\mbox{-\underline{\rm mod}}$.

 Wakamastu's construction is parallel to the one in \cite{TW1}, where the bimodule $_AT_B$ is assumed to have projective dimension at most one on both sides; also see \cite{TW2, Ta1}. The forerunners of  the work \cite{TW1} are  \cite{Ta}, \cite{AI} and \cite{Wak3}. On the other hand, if the bimodule $_AT_B$ is tilting of finite projective dimension, a triangle equivalence between these stable module categories was obtained in \cite{Ric}. Indeed, the result in \cite{Ric} is more general, which claims that such a triangle equivalence exists provided that $A$ and $B$ are derived equivalent.  However, the equivalence in \cite{Ric} is less explicit but with the advantage of being a triangle equivalence.

It is natural to ask whether Wakamatsu's equivalence is  a triangle equivalence. The aim of this paper is to answer this question affirmatively.

In Section 2, we recall basic facts on cotorsion pairs and $\partial$-functors. We prove that Wakamatsu's functor is a triangle functor in Section 3. In the last section, we recall the setting of \cite{Wak2},  where Wakamatsu's functor becomes a triangle equivalence.

We fix a commutative artinian ring $R$. Denote by $D={\rm Hom}_R(-, E)$ the Matlis duality, where $E$ is the minimal injective cogenerator of $R$. For an artin $R$-algebra $A$, we denote by $A\mbox{-mod}$ the category of finitely generated left $A$-modules. Any full subcategory of $A\mbox{-mod}$ is assumed to be closed under isomorphisms.  We identify right $A$-modules as left $A^{\rm op}$-modules, where $A^{\rm op}$ is the opposite algebra.

\section{Cotorsion pairs and $\partial$-functors}

In this section, we recall basic facts on cotorsion pairs and $\partial$-functors. We study special envelops of short exact sequences. The main references on cotorsion pairs are \cite{AR, GT}.

\subsection{Cotorsion pairs} Let $A$ be an artin $R$-algebra. Let $X={_AX}$ be an $A$-module. For a full subcategory $\mathcal{C}$  of $A\mbox{-mod}$, a $\mathcal{C}$-\emph{precover} of $X$ means a morphism $f\colon C\rightarrow X$ with $C\in \mathcal{C}$ such that any morphism $t\colon C'\rightarrow X$ with $C'\in \mathcal{C}$ factors through $f$, that is, $t=f\circ t'$ for some morphism $t'\colon C'\rightarrow C$. The subcategory $\mathcal{C}$ is said to be \emph{contravariantly finite}, if any module has a $\mathcal{C}$-precover. Dually, one has the notions of a $\mathcal{C}$-\emph{preenvelop} and a \emph{covariantly finite subcategory.}

We say that a subcategory $\mathcal{C}$ is \emph{finite}, if $\mathcal{C}={\rm add} Y$ for some module $_AY$. Here, ${\rm add}Y$ denotes the full subcategory formed by direct summands of finite direct sums of copies of $Y$. Observe that a finite subcategory is both contravariantly finite and covariantly finite.

Let $\mathcal{V}, \mathcal{W}$ be two full subcategories of $A\mbox{-mod}$. We denote by ${^{\perp_1}\mathcal{V}}$ the full subcategory formed by those modules $X$ satisfying ${\rm Ext}_A^1(X, V)=0$ for all $V\in \mathcal{V}$. Similarly, by ${^\perp\mathcal{V}}$ we mean the full subcategory formed by those modules $X$ satisfying ${\rm Ext}_A^i(X, V)=0$ for each $i\geq 1$ and $V\in \mathcal{V}$.

By a \emph{special $\mathcal{V}$-preenvelop} of an $A$-module  $X$, we mean a monomorphism $\alpha\colon X\rightarrow V$ with $V\in \mathcal{V}$ and its cokernel contained in ${^{\perp_1}\mathcal{V}}$. Then $\alpha$ is indeed a $\mathcal{V}$-preenvelop of $X$. This is obtained by applying ${\rm Hom}_A(-, V')$ to the exact sequence $0\rightarrow X\stackrel{\alpha}\rightarrow V\rightarrow {\rm Cok}\alpha\rightarrow 0$ for each $V'\in \mathcal{V}$. Dually, one has the notation $\mathcal{W}^{\perp_1}$, $\mathcal{W}^\perp$ and the notion of a \emph{special $\mathcal{W}$-precover}.

A \emph{cotorsion pair} $(\mathcal{W}, \mathcal{V})$ in $A\mbox{-mod}$ consists of two full subcategories satisfying $\mathcal{W}={^{\perp_1}\mathcal{V}}$ and $\mathcal{V}={\mathcal{W}^{\perp_1}}$, in which case, both $\mathcal{W}$ and $\mathcal{V}$ are closed under direct summands and extensions. A cotorsion pair $(\mathcal{W}, \mathcal{V})$ is \emph{complete} if every $A$-module  has a special $\mathcal{V}$-preenvelop, which is equivalent to the condition that each module has a special $\mathcal{W}$-precover; see \cite[Lemma 2.2.6]{GT}. A cotorsion pair $(\mathcal{W}, \mathcal{V})$ is \emph{hereditary} if ${\rm Ext}_A^i(W, V)=0$ for each $i\geq 1$, $W\in \mathcal{W}$ and $V\in \mathcal{V}$. In this case, we have $\mathcal{W}={^\perp \mathcal{V}}$ and $\mathcal{V}={\mathcal{W}^\perp}$.

The first part of the following result is due to \cite[Proposition 3.6]{AR}. We include a proof for completeness.

\begin{lem}\label{lem:sp}
Let $(\mathcal{W}, \mathcal{V})$ be a cotorsion pair which is complete and hereditary. Let $\xi\colon 0\rightarrow X_1\stackrel{f}\rightarrow X_2\stackrel{g}\rightarrow X_3\rightarrow 0$ be an exact sequence of modules. Take any special $\mathcal{V}$-preenvelop $\alpha_1$ and $\alpha_3$ of $X_1$ and $X_3$, respectively.  Then there is a commutative diagram with exact rows
\[\xymatrix{
\xi \colon & 0\ar[r] & X_1 \ar[d]^-{\alpha_1} \ar[r]^-{f} & X_2 \ar@{.>}[d]^-{\alpha_2} \ar[r]^-{g} & X_3 \ar[d]^-{\alpha_3}\ar[r] & 0 \\
\xi_V \colon  & 0\ar[r] & V_1 \ar@{.>}[r]^-{f_V} & V_2 \ar@{.>}[r]^-{g_V} & V_3\ar[r] & 0,
}\]
where $\alpha_2$ is a special $\mathcal{V}$-preenvelop of $X_2$. Moreover, given any morphism $t\colon V_1\rightarrow V$ in $\mathcal{V}$ satisfying $t\circ \alpha_1=0$, there exists a morphism $t'\colon V_2\rightarrow V$ satisfying $t=t'\circ f_V$ and $t'\circ \alpha_2=0$.
\end{lem}

We might call the exact sequence $\xi_V$ a \emph{special $\mathcal{V}$-preenvelop} of $\xi$.

\begin{proof}
By a pushout of $\xi$ along $\alpha_1$, we have the following commutative exact diagram
\[\xymatrix{
0\ar[r] & X_1 \ar[d]^-{\alpha_1} \ar[r]^-{f} & X_2 \ar@{.>}[d]^-{a} \ar[r]^-{g} & X_3 \ar@{=}[d]\ar[r] & 0 \\
0\ar[r] & V_1 \ar@{.>}[r]^-{f'} & E \ar@{.>}[r]^-{g'} & X_3\ar[r] & 0
}\]
Then $a$ is a monomorphism with ${\rm Cok}a={\rm Cok}\alpha_1$. Consider the exact sequence $0\rightarrow X_3\stackrel{\alpha_3} \rightarrow V_3\rightarrow W_3\rightarrow 0$. Since $\alpha_3$ is a special $\mathcal{V}$-preenvelop of $X_3$,  its cokernel $W_3$ lies in $\mathcal{W}$. By ${\rm Ext}^2_A(W_3, V_1)=0$, we have the following commutative exact diagram
\[\xymatrix{
 & & 0\ar[d]& 0\ar[d]\\
0\ar[r] & V_1 \ar@{=}[d] \ar[r]^-{f'} & E \ar@{.>}[d]^-{a'} \ar[r]^-{g'} & X_3 \ar[d]^-{\alpha_3}\ar[r] & 0\\
0\ar[r] & V_1 \ar@{.>}[r] & V_2 \ar@{.>}[d] \ar@{.>}[r] & V_3\ar[d] \ar[r] & 0\\
        &                 & W_3 \ar[d] \ar@{=}[r] & W_3 \ar[d]     \\
         &                &  0 & 0
}\]
Then $V_2$ lies in $\mathcal{V}$. Put $\alpha_2=a'\circ a$, which is a special $\mathcal{V}$-preenvelop, since its cokernel lies in $\mathcal{W}$.

For the last statement, we consider the exact sequence  $0\rightarrow W_1 \stackrel{f_W}\rightarrow W_2\rightarrow W_3\rightarrow 0$ of the cokernels of $\alpha_i$'s. Then $t=\bar{t}\circ \pi_1$ for some morphism $\bar{t}\colon W_1\rightarrow V$, where $\pi_1\colon V_1\rightarrow W_1$ is the canonical projection. Since ${\rm Ext}_A^1(W_3, V)=0$, we infer that $\bar{t}$ factors through $f_W$, that is, $\bar{t}=t''\circ f_W$ for some morphism $t''\colon W_2\rightarrow V$. Set $t'=t''\circ \pi_2$ with $\pi_2\colon V_2\rightarrow W_2$ the canonical projection. Then we are done.
\end{proof}

The following result indicates that taking the  special $\mathcal{V}$-preenvelop of a  short exact sequence
 is partially functorial.

\begin{lem}\label{lem:com}
Let $(\mathcal{W}, \mathcal{V})$ be a cotorsion pair which is complete and hereditary. Assume that we are given the top of the following diagram, which is commutative with rows being short exact sequences. Consider their special $\mathcal{V}$-preenvelops as in the previous lemma. Here, the morphisms $\alpha_i\colon X_i\rightarrow V_i$ and $\alpha'_i\colon Y_i\rightarrow V'_i$ are the special $\mathcal{V}$-preenvelops. Then the dotted morphisms exist, which make the diagram commute.
\[\xymatrix@!=7pt{
& Y_1\ar[rr]^-{h} \ar[dd]&& Y_2\ar[rr]^-{k} \ar[dd]&& Y_3 \ar[dd] \\
 X_1 \ar[dd] \ar[ur]|{a}\ar[rr]^(0.42){f} && X_2 \ar[dd]\ar[ur]|{b}\ar[rr]^(0.42){g} && X_3 \ar[dd] \ar[ur]|{c}\\
&   V'_1\ar[rr]^(0.43){h_V} && V'_2\ar[rr]^(0.43){k_V} && V'_3 \\
  V_1 \ar@{.>}[ur]|{a_V}\ar[rr]^-{f_V} && V_2 \ar@{.>}[ur]|{b_V} \ar[rr]^-{g_V} && V_3 \ar@{.>}[ur]|{c_V}
}\]
\end{lem}

\begin{proof}
By the special $\mathcal{V}$-preenvelop $\alpha_1$, we have a morphism $a_V\colon V_1\rightarrow V'_1$ satisfying $\alpha'_1\circ a=a_V\circ \alpha_1$. For the same reason, we have $b'_V\colon V_2\rightarrow V'_2$ satisfying $\alpha'_2\circ b=b'_V\circ \alpha_2$. But, in general, $h_V\circ a_V\neq b'_V\circ f_V$. By a diagram-chasing, we do have $(h_V\circ a_V-b'_V\circ f_V)\circ\alpha_1=0$. By Lemma \ref{lem:sp} there is a morphism $t'\colon V_2\rightarrow V'_2$ such that $t'\circ f_V=h_V\circ a_V-b'_V\circ f_V$ and $t'\circ \alpha_2=0$. Set $b_V=t'+b'_V$. Then there is a unique morphism $c_V$ such that $c_V\circ g_V=k_V\circ b_V$. By a diagram-chasing, we obtain $\alpha'_3\circ c=c_V\circ \alpha_3$. Then we are done.
\end{proof}

\subsection{Stable categories and $\partial$-functors}

Let $\mathcal{A}$ be an abelian category. Recall that it is a \emph{Frobenius category} provided that it has enough projectives and enough injectives such that the class of projective objects coincides with the class of injective objects. The \emph{stable categor}y $\underline{\mathcal{A}}$ modulo projectives is defined as follows: the objects are the same as $\mathcal{A}$; for two objects $X, Y$, the Hom group, denoted by ${\underline{\rm Hom}}_\mathcal{A}(X, Y)$, is defined to be the quotient group ${\rm Hom}_\mathcal{A}(X, Y)/P(X, Y)$, where $P(X, Y)$ denotes the subgroup formed by morphisms that factor through projectives; the composition of morphisms is induced from $\mathcal{A}$. For a morphism $f\colon X\rightarrow Y$ in $\mathcal{A}$, we denote by $\underline{f}\colon X\rightarrow Y$ the corresponding morphism in $\underline{\mathcal{A}}$.

For a Frobenius category $\mathcal{A}$, its stable category $\underline{\mathcal{A}}$ has a natural triangulated structure.
For the translation functor $\Sigma$, we \emph{fix} for each object $X$ an exact sequence $0\rightarrow X\stackrel{i_X} \rightarrow I(X) \stackrel{d_X}\rightarrow \Sigma(X)\rightarrow 0$ with $I(X)$ injective. Any exact sequence $0\rightarrow X\stackrel{f} \rightarrow Y\stackrel{g}\rightarrow Z\rightarrow 0$ in $\mathcal{A}$ yields an exact triangle $X\stackrel{\underline{f}} \rightarrow Y\stackrel{\underline{g}}\rightarrow Z\stackrel{\underline{\omega}}\rightarrow \Sigma(X)$, where $\omega$ is given by the following commutative diagram
\begin{align}\label{equ:3}
\xymatrix{
0 \ar[r] & X \ar@{=}[d] \ar[r]^-{f} & Y \ar@{.>}[d]\ar[r]^-{g} & Z \ar[r] \ar@{.>}[d]^-{\omega} & 0\\
0 \ar[r] & X \ar[r]^-{i_X} & I(X) \ar[r]^-{d_X} & \Sigma(X) \ar[r] & 0.}
\end{align}
Here, we use the injectivity of $I(X)$. The morphism $\omega$ is not unique, but its image $\underline{\omega}$ in $\underline{\mathcal{A}}$ is unique. In particular, for a selfinjective algebra $A$, its stable module category $A\mbox{-\underline{\rm mod}}$ becomes a triangulated category. For details, we refer to \cite[I.2]{Hap}.

Let $F\colon \mathcal{A}\rightarrow \mathcal{T}$ be an additive functor from an abelian category to a triangulated category. The translation functor on $\mathcal{T}$ is denoted by $\Sigma$.  Following \cite[Section 1]{Kel}, we say that $F$ is a \emph{$\partial$-functor} provided that for each short exact sequence $\xi\colon 0\rightarrow X\stackrel{f}\rightarrow Y\stackrel{g}\rightarrow Z\rightarrow 0$ in $\mathcal{A}$, there is a chosen morphism $\omega_\xi\colon F(Z)\rightarrow \Sigma(FX)$,  which fits into an exact triangle $F(X)\stackrel{F(f)}\rightarrow F(Y)\stackrel{F(g)}\rightarrow F(Z)\stackrel{\omega_\xi}\rightarrow \Sigma(FX)$. Moreover, the chosen morphism $\omega_\xi$ is functorial in $\xi$. More precisely, for each commutative exact diagram
\[\xymatrix{
\xi\colon & 0 \ar[r] & X \ar[d]^-{a}\ar[r]^-{f} & Y \ar[d]^-{b}\ar[r]^-{g} & Z \ar[d]^-{c}\ar[r] & 0\\
\xi'\colon & 0 \ar[r] & X'\ar[r]^-{f'} & Y' \ar[r]^-{g'} & Z'\ar[r] & 0,
}\]
there is a morphism between exact triangles
\[\xymatrix{
F(X) \ar[d]^-{F(a)}\ar[r]^-{F(f)} & F(Y) \ar[d]^-{F(b)}\ar[r]^-{F(g)} & F(Z) \ar[d]^-{F(c)}\ar[r]^-{\omega_\xi} & \Sigma(FX)\ar[d]^-{\Sigma(Fa)}\\
 F(X')\ar[r]^-{F(f')} & F(Y') \ar[r]^-{F(g')} & F(Z')\ar[r]^-{\omega_{\xi'}} & \Sigma(FX').
}\]
Indeed, it suffices to  verify that the rightmost square commutes. We observe that for a Frobenius category, the canonical functor $\mathcal{A}\rightarrow \underline{\mathcal{A}}$ is a $\partial$-functor.

The following fact is well known.

\begin{lem}{\rm (\cite[Lemma 2.5]{Chen})}\label{lem:par}
Let $\mathcal{A}$ be a Frobenius category and $F\colon \mathcal{A}\rightarrow \mathcal{T}$ be a $\partial$-functor which vanishes on projective objects. Then the induced functor $F\colon \underline{\mathcal{A}}\rightarrow \mathcal{T}$ is a triangle functor. \hfill $\square$
\end{lem}

\section{Wakamatsu's functor}
In this section, we first recall from \cite{Wak2} the construction of Wakamatsu's functor. We will prove in Theorem \ref{thm:main} that it is a triangle functor.

\subsection{The construction} Let $A$ and $B$ be two artin $R$-algebras. Let $_AT_B$ be an $A$-$B$-bimodule, on which $R$ acts centrally.

We denote by $\varepsilon$ and $\eta$ the counit and unit of the adjoint pair $(T\otimes_B-, {\rm Hom}_A(T, -))$ on $A\mbox{-mod}$ and $B\mbox{-mod}$, respectively. More precisely, for each $A$-module $X$, the map $\varepsilon_X\colon T\otimes_B {\rm Hom}_A(T, X)\rightarrow X$ is defined by $\varepsilon_X(t\otimes f)=f(t)$; for each $B$-module $Y$, the map $\eta_Y\colon Y\rightarrow {\rm Hom}_A(T, T\otimes_BY)$ is given by $\eta_Y(y)(t)=t\otimes y$.

From now on, we assume that the $A$-$B$-bimodule $T$ is \emph{faithfully balanced}, that is, the structure maps $A\rightarrow {\rm End}_{B^{\rm op}}(T)$ and $B^{\rm op}\rightarrow {\rm End}_A(T)$ are both isomorphisms. In this case, we have two canonical bimodule isomorphisms
$$\delta\colon DT\otimes_A T\stackrel{\sim}\longrightarrow DB, \; {\rm and} \quad \delta'\colon T\otimes_B DT\stackrel{\sim}\longrightarrow DA,$$
which are given by $\delta(f\otimes t)(b)=f(tb)$ and $\delta'(t\otimes f)(a)=f(at)$. Here, $DT$ has the induced $B$-$A$-bimodule structure.

Recall that $T(A)=A\oplus DA$ is the trivial extension of $A$; it is a symmetric algebra, and thus selfinjecive. A $T(A)$-module is identified with a pair $(X, \phi)$, where $X$ is an $A$-module and the \emph{structure map} $\phi\colon DA\otimes X\rightarrow X$ is an $A$-module morphism satisfying $\phi \circ (DA\otimes \phi)=0$. We sometimes suppress $\phi$ and denote the pair by $X$. Similar notation applies to $T(B)$-modules.

For an $A$-module $_AV$, we consider the $B$-module
$$L(V)={\rm Hom}_A(T, V)\oplus (DT\otimes_AV),$$
 whose elements are viewed as column vectors. Then $L(V)$ becomes a $T(B)$-module via the structure map
$$\begin{pmatrix} 0 & 0\\
                * & 0\end{pmatrix}\colon DB\otimes_B L(V)\longrightarrow L(V),$$
                where $*$ is given by the composition $(DT\otimes \varepsilon_V)\circ (\delta^{-1}\otimes {\rm Hom}_A(T, V))$.  We observe that $L(T)$ is isomorphic to the regular module $T(B)$.

For a $T(A)$-module $X=(X, \phi)$, the $B$-module $DT\otimes_A X$ becomes a $T(B)$-module via the structure map $DB\otimes _B (DT\otimes_A X)\rightarrow DT\otimes_A X$, which is given by the composition
$$-(DT\otimes \phi)\circ (DT\otimes \delta'\otimes X) \circ (\delta^{-1}\otimes DT\otimes_A X).$$
Here, the minus sign is needed in the following construction.

We assume that $(\mathcal{W}, \mathcal{V})$ is a  complete cotorsion pair in $A\mbox{-mod}$ such that $\mathcal{W}\cap \mathcal{V}={\rm add} T$. For each $A$-module $X$, we \emph{fix} a special $\mathcal{V}$-preenvelop $\alpha_X\colon X\rightarrow V(X)$ once and for all.

We observe that $DT\otimes \alpha_X$ is always injective by
$${\rm Tor}_1^A(DT, {\rm Cok}\alpha_X)\simeq D{\rm Ext}^1_A({\rm Cok}\alpha_X, T)=0.$$
For a morphism $f\colon X\rightarrow X'$, there is a morphism $f_V\colon V(X)\rightarrow V(X')$ satisfying $\alpha_{X'}\circ f=f_V\circ \alpha_X$. Note that the morphism $f_V$ is not unique.

We recall from \cite[Section 1]{Wak2} the construction of Wakamatsu's functor $$\mathcal{S}\colon T(A)\mbox{-mod}\longrightarrow T(B)\mbox{-\underline{\rm mod}}.$$
 By \cite[Lemma 1.1]{Wak2} and \cite[Proposition 1.5]{TW1}, for each $T(A)$-module $X=(X, \phi)$, we have the following injective $T(B)$-module homomorphism
\begin{align}\label{equ:1}
\begin{pmatrix}
\Delta_X
\\
DT\otimes \alpha_X\end{pmatrix}\colon DT\otimes_A X \longrightarrow L(V(X))={\rm Hom}_A(T, V(X))\oplus (DT\otimes_AV(X)),
\end{align}
where $\Delta_X$ is given by the composition
$${\rm Hom}_A(T, -\alpha_X\circ \phi \circ (\delta'\otimes X)) \circ \eta_{(DT\otimes_AX)}.$$
We define $\mathcal{S}(X)$ to be the cokernel of this monomorphism. This notation is somehow sloppy, since $\mathcal{S}(X)$ depends on $(X, \phi)$, not just the underlying $A$-module $X$.

For a morphism $f\colon (X, \phi)\rightarrow (X', \phi')$ of $T(A)$-modules, we take any morphism $f_V\colon V(X)\rightarrow V(X')$ satisfying $\alpha_{X'}\circ f=f_V\circ \alpha_X$. Then the left square in the following diagram commutes.
\begin{align}\label{equ:2}
\xymatrix{
0\ar[r] &  DT\otimes_A X\ar[d]^-{DT\otimes f}\ar[r] & L(V(X)) \ar[d]^-{L(f_V)} \ar[r] & \mathcal{S}(X) \ar@{.>}[d]^-{\mathcal{S}(f)}\ar[r] & 0\\
0\ar[r] &  DT\otimes_A X'\ar[r] & L(V(X')) \ar[r] & \mathcal{S}(X') \ar[r] & 0  }
\end{align}
Then there is a unique morphism $\mathcal{S}(f)$ making the diagram commute. However, the morphism $\mathcal{S}(f)$ depends on the choice of $f_V$, but its image $\underline{\mathcal{S}(f)}$ in the stable category $T(A)\mbox{-\underline{mod}}$ is independent of the choice. This completes the construction of Wakamatsu's functor $\mathcal{S}$; see \cite[Lemma 1.2]{Wak2}. Since the functor $\mathcal{S}$ depends on the cotorsion pair $(\mathcal{W}, \mathcal{V})$, we will say that $\mathcal{S}$ is \emph{associated} to  $(\mathcal{W}, \mathcal{V})$.

In what follows, when we write $\mathcal{S}(f)$, we mean the corresponding morphism in $T(B)\mbox{-mod}$. We have to keep in mind that $\mathcal{S}(f)$ depends on the choice of $f_V$, not just $f$.

The following subtlety has to be clarified; compare the treatment in the third paragraph of \cite[p.19]{Hap}. Assume that we are given a special $\mathcal{V}$-preenvelop $\alpha'_X\colon X\rightarrow V'(X)$, which might not equal  the fixed $\alpha_X$. Replacing $V(X)$ by $V'(X)$ and $\alpha_X$ by $\alpha'_X$ in (\ref{equ:1}), we obtain the cokernel $\mathcal{S}'(X)$. Then we have a \emph{canonical isomorphism} in $T(B)\mbox{-\underline{\rm mod}}$
\begin{align}\label{equ:can}
{\rm can}\colon \mathcal{S}'(X)\stackrel{\sim}\longrightarrow \mathcal{S}(X).
\end{align}
Indeed, there is a morphism $s\colon V'(X)\rightarrow V(X)$ satisfying $\alpha_X=s\circ \alpha'_X$. Then a similar diagram as (\ref{equ:2}) defines the above isomorphism, which is independent of the choice of $s$. Consider the previous morphism $f\colon (X, \phi)\rightarrow (X', \phi')$. There is a morphism $f_{V'}\colon V'(X)\rightarrow V(X')$ satisfying $f_{V'}\circ\alpha'_X=\alpha_{X'}\circ f$. Then by replacing $f_V$ by $f_{V'}$ in (\ref{equ:2}), we obtain a morphism
\begin{align*}
\mathcal{S}'(f)\colon \mathcal{S}'(X)\longrightarrow \mathcal{S}(X'),
 \end{align*}
 which depends on the choice of $f_{V'}$. We observe the following fact
 \begin{align}\label{equ:4}
 \underline{\mathcal{S}'(f)}=\underline{\mathcal{S}(f)}\circ {\rm can}.
 \end{align}
 This fact enables us to abuse $\mathcal{S}'(X)$ with $S(X)$, $\underline{\mathcal{S}'(f)}$ with $\underline{\mathcal{S}(f)}$. The notation $\mathcal{S}'(f)$ also applies, if the range $X'$ of $f$ has taken a special $\mathcal{V}$-preenvelop, different from the fixed one.

\subsection{The $\partial$-functor}

The above recalled Wakamatsu's functor $\mathcal{S}$ vanishes on projective $T(A)$-modules; see \cite[Lemma 1.3]{Wak2}. Then it induces an additive functor from the stable module category of $T(A)$ to that of  $T(B)$. Our main result claims that the induced functor is a triangle functor, provided that the cotorsion pair $(\mathcal{W}, \mathcal{V})$ is in addition hereditary.

\begin{thm}\label{thm:main}
Let $_AT_B$ is a faithfully balanced $A$-$B$-bimodule. Assume that $(\mathcal{W}, \mathcal{V})$ is a complete hereditary cotorsion pair in $A\mbox{-{\rm mod}}$ satisfying $\mathcal{W}\cap \mathcal{V}={\rm add} T$.  Then Wakamatsu's functor $\mathcal{S}\colon T(A)\mbox{-{\rm mod}}\rightarrow T(B)\mbox{-\underline{\rm mod}}$ associated to  $(\mathcal{W}, \mathcal{V})$ is a $\partial$-functor. In particular, it induces a triangle functor $T(A)\mbox{-\underline{\rm mod}}\rightarrow T(B)\mbox{-\underline{\rm mod}}$.
\end{thm}

\begin{proof}
The second statement follows from Lemma \ref{lem:par}.  For the first statement, we take an exact sequence $\xi\colon 0\rightarrow (X_1, \phi_1)\stackrel{f}\rightarrow (X_2, \phi_2) \stackrel{g}\rightarrow (X_3, \phi_3)\rightarrow 0$ in $T(A)\mbox{-mod}$. Recall the fixed special $\mathcal{V}$-preenvelops $\alpha_{X_1}\colon X_1\rightarrow V(X_1)$ and $\alpha_{X_3}\colon X_3\rightarrow V(X_3)$. Applying Lemma \ref{lem:sp}, we obtain the following commutative exact diagram
\[\xymatrix{
 0\ar[r] & X_1 \ar[d]^-{\alpha_{X_1}} \ar[r]^-{f} & X_2 \ar@{.>}[d]^-{\alpha'_{X_2}} \ar[r]^-{g} & X_3 \ar[d]^-{\alpha_{X_3}}\ar[r] & 0 \\
0\ar[r] & V(X_1) \ar@{.>}[r]^-{f_{V'}} & V'(X_2) \ar@{.>}[r]^-{g_{V'}} & V(X_3)\ar[r] & 0.
}\]
Here, $\alpha'_{X_2}$ is a special $\mathcal{V}$-preenvelop, which might not equal the fixed $\alpha_{X_2}$. For this reason, we use the notation $f_{V'}$ and $g_{V'}$, instead of $f_V$ and $g_V$.

We have the following commutative diagram in $T(B)\mbox{-mod}$ with exact rows.
\[\xymatrix{
0\ar[r] & DT\otimes_A X_1 \ar[d]^-{DT\otimes f}\ar[r] & L(V(X_1)) \ar[d]^-{L(f_{V'})}\ar[r] & \ar@{.>}[d]^-{\mathcal{S}'(f)}\mathcal{S}(X_1)\ar[r] & 0\\
0\ar[r] & DT\otimes_A X_2 \ar[d]^-{DT\otimes g}\ar[r] & L(V'(X_2)) \ar[d]^-{L(g_{V'})} \ar[r] & \mathcal{S}'(X_2)  \ar@{.>}[d]^-{\mathcal{S}'(g)} \ar[r] & 0\\
0\ar[r] & DT\otimes_A X_3 \ar[r] & L(V(X_3)) \ar[r] & \mathcal{S}(X_3)\ar[r] & 0
}\]
Here, for the notation $\mathcal{S}'(f)$ and $\mathcal{S}'(g)$, we refer to the last paragraph in the previous subsection.

We claim that the sequence $0\rightarrow \mathcal{S}(X_1)\stackrel{\mathcal{S}'(f)}\rightarrow \mathcal{S}'(X_2)\stackrel{\mathcal{S}'(g)}\rightarrow \mathcal{S}(X_3)\rightarrow 0$ is exact. For the claim, we view the columns in the above diagram as complexes. The middle complex is written as ${\rm Hom}_A(T, V(X_\bullet))\oplus (DT\otimes_A V(X_\bullet))$. Since ${\rm Ext}_A^1(T, V(X_1))=0$, the subcomplex ${\rm Hom}_A(T, V(X_\bullet))$ is acyclic. The claim is equivalent to the fact that the following monomorphism
$$DT\otimes \alpha_{X_\bullet}\colon DT\otimes_A X_\bullet \longrightarrow DT\otimes_AV(X_\bullet)$$
is a quasi-isomorphism. However, the cokernel of $DT\otimes \alpha_{X_\bullet}$ is isomorphic to $DT\otimes W(X_\bullet)$, where each $W(X_i)$ is the cokernel of $\alpha_{X_i}$, respectively. Here, we abuse $W(X_2)$ with $W'(X_2)$, the cokernel of $\alpha'_{X_2}$. The cokernels $W(X_i)$ belong to $\mathcal{W}$ and the complex $W(X_\bullet)$ is acyclic. It follows that the complex $DT\otimes W(X_\bullet)$ is also cyclic, since ${\rm Tor}^A_1(DT, W(X_3))\simeq D{\rm Ext}_A^1(W(X_3), T)=0$. From this, we infer that $DT\otimes \alpha_{X_\bullet}$ is a quasi-isomorphism. We are done with the claim.

Thanks to the claim and (\ref{equ:3}), we have an exact triangle in $T(B)\mbox{-\underline{mod}}$
$$\mathcal{S}(X_1)\xrightarrow{\underline{\mathcal{S}'(f)}} \mathcal{S}'(X_2)\xrightarrow{\underline{\mathcal{S}'(g)}} \mathcal{S}(X_3)\stackrel{\omega_\xi}\longrightarrow \Sigma(\mathcal{S}X_1).$$
Identifying $\mathcal{S}'(X_2)$ with $\mathcal{S}(X_2)$ via the canonical isomorphism (\ref{equ:can}) and using (\ref{equ:4}), we obtain the desired triangle
$$\mathcal{S}(X_1)\xrightarrow{\underline{\mathcal{S}(f)}} \mathcal{S}(X_2)\xrightarrow{\underline{\mathcal{S}(g)}} \mathcal{S}(X_3)\stackrel{\omega_\xi}\longrightarrow \Sigma(\mathcal{S}X_1).$$

It remains to show that $\omega_\xi$ is functorial in $\xi$. Before doing this, we notice that $\omega_{\xi}$ seems to depend on our choice of $\alpha'_{X_2}$, $f_{V'}$ and $g_{V'}$. We claim that $\omega_\xi$ is actually independent of the choice. This will be proved along with the functorial property of $\omega_{\xi}$.

We assume that there is a commutative diagram in $T(A)\mbox{-mod}$ with exact rows
\[\xymatrix{
\xi\colon & 0 \ar[r] & (X_1, \phi_1) \ar[d]^-{a}\ar[r]^-{f} & (X_2, \phi_2) \ar[d]^-{b}\ar[r]^-{g}\ar[r] & (X_3, \phi_3) \ar[d]^-{c}\ar[r] &  0\\
\xi'\colon & 0 \ar[r] & (Y_1, \psi_1) \ar[r]^-{h} & (Y_2, \psi_2) \ar[r]^-{k} & (Y_3, \psi_3) \ar[r] & 0.
}\]
For $\xi'$, we have the following commutative diagram
\[\xymatrix{
 0\ar[r] & Y_1 \ar[d]^-{\alpha_{Y_1}} \ar[r]^-{h} & Y_2 \ar@{.>}[d]^-{\alpha'_{Y_2}} \ar[r]^-{k} & Y_3 \ar[d]^-{\alpha_{Y_3}}\ar[r] & 0 \\
0\ar[r] & V(Y_1) \ar@{.>}[r]^-{h_{V'}} & V'(Y_2) \ar@{.>}[r]^-{k_{V'}} & V(Y_3)\ar[r] & 0,
}\]
which yields an exact sequence $0\rightarrow \mathcal{S}(Y_1)\stackrel{\mathcal{S}'(h)}\rightarrow \mathcal{S}'(Y_2)\stackrel{\mathcal{S}'(k)}\rightarrow \mathcal{S}(y_3)\rightarrow 0$ of $T(B)$-modules. We apply Lemma \ref{lem:com} to obtain the relevant morphisms $a_V\colon V(X_1)\rightarrow V(Y_1)$, $b_{V'}\colon V'(X_2)\rightarrow V'(Y_2)$ and $c_V\colon V(X_3)\rightarrow V(Y_3)$, which make the diagram commute. Then we obtain a commutative diagram between two $3\times 3$ diagrams in $T(B)\mbox{-mod}$, which yields a commutative exact diagram
\[\xymatrix{
0\ar[r] & \mathcal{S}(X_1) \ar@{.>}[d]^-{\mathcal{S}(a)} \ar[r]^-{\mathcal{S}'(f)} &  \mathcal{S}'(X_2)\ar@{.>}[d]^-{\mathcal{S}'(b)} \ar[r]^-{\mathcal{S}'(g)} &  \mathcal{S}(X_3) \ar@{.>}[d]^-{\mathcal{S}(c)}\ar[r] & 0\\
0\ar[r] & \mathcal{S}(Y_1) \ar[r]^-{\mathcal{S}'(h)} &  \mathcal{S}'(Y_2) \ar[r]^-{\mathcal{S}'(k)} &  \mathcal{S}(Y_3) \ar[r] & 0.
}\]
Recall that the canonical functor $T(B)\mbox{-mod}\rightarrow T(B)\mbox{-\underline{\rm mod}}$ is a $\partial$-functor. In particular, the above diagram implies a morphism between exact triangles in $T(B)\mbox{-\underline{\rm mod}}$. Then we have
$$\Sigma (\underline{\mathcal{S}a})\circ \omega_\xi=\omega_{\xi'}\circ \underline{\mathcal{S}(c)}.$$
This proves the functorialness of $\omega_\xi$. In particular, if all of $a$, $b$ and $c$ are  the identity maps, this proves that $\omega_\xi$ is independent of our choice.
\end{proof}

\section{Wakamatsu-tilting bimodules}

In this section, we recall from \cite{Wak1, MR} basic facts on Wakamatsu-tilting bimodules. For a certain Wakamatsu-tilting bimodule, Wakamatsu's functor in Theorem \ref{thm:main} can be defined and becomes a triangle equivalence.

Let $_AT$ be an $A$-module satisfying ${\rm Ext}_A^i(T, T)=0$ for each $i\geq 1$. Write $T^\perp$ for the full subcategory consisting of those modules $X$ satisfying ${\rm Ext}_A^i(T, X)=0$ for each $i\geq 1$. Set $_T\mathcal{X}$ to be the full subcategory formed by those modules $X$, which admit a long exact sequence $\cdots\rightarrow T^{-2} \stackrel{d^{-2}}\rightarrow T^{-1}\stackrel{d^{-1}}\rightarrow T^0\rightarrow X\rightarrow 0 $ with each $T^{-i}\in {\rm add} T$ and each cokernel ${\rm Cok}d^{-i}\in T^\perp$. In particular, $_T\mathcal{X}\subseteq T^\perp$. Recall from \cite[Proposition 5.1]{AR} that $_T\mathcal{X}$ is closed under extensions, cokernels of monomorphisms and direct summands. Similarly, we have the subcategories $\mathcal{X}_T\subseteq {^\perp T}$.

Let $_AT_B$ be an $A$-$B$-bimodule. We say that $_AT_B$ is a \emph{Wakamatsu-tilting bimodule} provided that it is faithfully balanced satisfying ${\rm Ext}_A^i(T, T)=0={\rm Ext}^i_{B^{\rm op}}(T, T)$ for each $i\geq 1$. An $A$-module $_AT$ is a \emph{Wakamatsu-tilting module} if the natural bimodule $_AT_B$ is Wakamatsu-tilting with $B={\rm End}_A(T)^{\rm op}$. In this case, the dual bimodule $_B(DT)_A$ is also Wakamatsu-tilting, and thus the $B$-module $_B(DT)$ is Wakamatsu-tilting.

We collect known facts on Wakamatsu-tilting bimodules in the following lemma.

\begin{lem}\label{lem:Wak}
Let $_AT_B$ be a Wakamatsu-tilting bimodule. Then the following statements hold.
\begin{enumerate}
\item $(^\perp({_T\mathcal{X}}), {_T\mathcal{X}})$ and $(\mathcal{X}_T, (\mathcal{X}_T)^\perp)$ are both hereditary cotorsion pairs in $A\mbox{-{\rm mod}}$; moreover, ${^\perp({_T\mathcal{X}})}\cap {_T\mathcal{X}}={\rm add}T= \mathcal{X}_T\cap (\mathcal{X}_T)^\perp$ and $(\mathcal{X}_T)^\perp\subseteq {_T\mathcal{X}}$.
\item $(\mathcal{X}_{DT}, (\mathcal{X}_{DT})^\perp)$ and $(^\perp(_{DT}\mathcal{X}), {_{DT}\mathcal{X}})$  are both hereditary cotorsion pairs in $B\mbox{-{\rm mod}}$; moreover, $\mathcal{X}_{DT}\cap (\mathcal{X}_{DT})^\perp={\rm add} DT= {^\perp(_{DT}\mathcal{X})}\cap {_{DT}\mathcal{X}}$ and $^\perp(_{DT}\mathcal{X})\subseteq {\mathcal{X}_{DT}}$.
    \item There are equivalences between these subcategories given by the Hom and tensor functors.
    \[\xymatrix{
    {_T\mathcal{X}} \ar@<+.5ex>[d]^-{{\rm Hom}_A(T, -)}& \supseteq  &(\mathcal{X}_T)^\perp \ar@<+.5ex>[d]^-{{\rm Hom}_A(T, -);};  & &    \mathcal{X}_T \ar@<-.5ex>[d]_{DT\otimes_A-} & \supseteq  & {^\perp({_T\mathcal{X}})} \ar@<-.5ex>[d]_{DT\otimes_A-}\\
    \mathcal{X}_{DT} \ar@<+.5ex>[u]^-{T\otimes_B-} & \supseteq &  {^\perp({_{DT}\mathcal{X}})} \ar@<+.5ex>[u]^-{T\otimes_B-};   &  &  {_{DT}\mathcal{X}} \ar@<-.5ex>[u]_{{\rm Hom}_B(DT, -)} & \supseteq & (\mathcal{X}_{DT})^\perp\ar@<-.5ex>[u]_{{\rm Hom}_B(DT, -)}
    }\]
\end{enumerate}
\end{lem}

In general, the above cotorsion pairs are not complete.

\begin{proof}
For (1), we refer to \cite[Proposition 3.1]{MR}, and (2) follows from (1) applied to the dual bimodule $_B(DT)_A$.   For (3), we refer to \cite[Proposition 2.14]{Wak2}.
\end{proof}

Following \cite{Wei}, a Wakamatsu-tilting bimodule $_AT_B$ is \emph{good} provided that there are  cotorsion pairs $(\mathcal{W}, \mathcal{V})$ in $A\mbox{-mod}$ and $(\mathcal{Y}, \mathcal{Z})$ in $B\mbox{-mod}$, respectively, which satisfy the following conditions.
\begin{enumerate}
\item[(GW1)] These two cotorsion pairs are complete hereditary.
\item[(GW2)] $\mathcal{W}\cap \mathcal{V}={\rm add} T$ and $\mathcal{Y}\cap \mathcal{Z}={\rm add} DT$.
\item[(GW3)] The adjoint pair $(T\otimes_B-, {\rm Hom}_A(T, -))$ induces an equivalence $\mathcal{V}\stackrel{\sim}\longrightarrow \mathcal{Y}$.
\item[(GW4)] The adjoint pair $(DT\otimes_A-, {\rm Hom}_B(DT, -))$ induces an equivalence $\mathcal{W}\stackrel{\sim}\longrightarrow \mathcal{Z}$.
\end{enumerate}
We mention that these conditions are essentially  given in \cite[Hypothesis 1.4]{Wak2}. In the above situation, we observe that $(\mathcal{X}_T)^\perp\subseteq \mathcal{V}\subseteq {_T\mathcal{X}}$. Indeed, one proves that $_AT$ is an Ext-projective generator for $\mathcal{V}$ and then applies \cite[Corollary 3.3]{MR}; also see \cite[Proposition 3.2.2]{Wei}.

In  the following example, we use the well-known fact: a cotorsion pair $(\mathcal{C}, \mathcal{D})$ is complete if and only if $\mathcal{D}$ is covariantly finite, if and only if $\mathcal{C}$ is contravariantly finite; see \cite[Proposition 1.9]{AR}.

\begin{exm}
{\rm Let $_AT_B$ be a Wakamatsu-tilting bimodule.

(1) If both $_AT$ and $T_B$ have finite projective dimension, then $_AT_B$ is called a \emph{tilting bimodule}. This coincides with the tilting module of finite projective dimension in \cite{Miy,CPS, Hap}. In this case, the cotorsion pairs $(^\perp({_T\mathcal{X}}), {_T\mathcal{X}})$ and $(\mathcal{X}_{DT}, (\mathcal{X}_{DT})^\perp)$ are complete; see \cite[Theorem 2.17]{Wak2}. In this case, we have ${_T\mathcal{X}}=T^\perp$; see \cite[Theorem 5.4]{AR}. Hence, by Lemma \ref{lem:Wak} a tilting bimodule is  a good Wakamatsu-tilting bimodule.

If both $_AT$ and $T_B$ have finite injective dimension, then $_AT_B$ is called a \emph{cotilting bimodule}. By duality, we observe that a cotilting  bimodule is a good Wakamatsu-tilting bimodule.

(2) Following \cite{Wei}, the Wakamatsu-tilting bimodule $_AT_B$ is said to be of \emph{finite type}, if either the subcategory ${^\perp({_T\mathcal{X}})}$ or $(\mathcal{X}_T)^\perp$ of $A\mbox{-{\rm mod}}$ is finite. This happens when $A$ or $B$ is of finite representation type; for an explicit example, see \cite[Example 3.1]{Wak2}. Then a Wakamatsu-tilting bimodule of finite type is good.

Indeed, if ${^\perp({_T\mathcal{X}})}$ is finite, so is $(\mathcal{X}_{DT})^\perp$ by Lemma \ref{lem:Wak}(3). Then both cotorsion pairs $(^\perp({_T\mathcal{X}}), {_T\mathcal{X}})$ and $(\mathcal{X}_{DT}, (\mathcal{X}_{DT})^\perp)$ are complete. Similar argument applies if $(\mathcal{X}_T)^\perp$ is finite.}
\end{exm}

We now reformulate Wakamatsu's equivalence as follows, which combines \cite[Theorem 1.5]{Wak2} and Theorem \ref{thm:main}.

\begin{thm} {\rm (Wakamatsu)}\label{thm:Wak}
Let $_AT_B$ be a good Wakamatsu-tilting bimodule with the relevant cotorsion pairs $(\mathcal{W}, \mathcal{V})$ and $(\mathcal{Y}, \mathcal{Z})$ as above. Then the Wakamatsu's functor $$\mathcal{S}\colon T(A)\mbox{-\underline{\rm mod}}\longrightarrow T(B)\mbox{-\underline{\rm mod}}$$
 associated to $(\mathcal{W}, \mathcal{V})$ is a triangle equivalence. \hfill $\square$
\end{thm}

\begin{rem}
We keep the assumptions in Theorem \ref{thm:Wak}.

(1) If the given Wakamatsu-tilting bimodule $_AT_B$ is tilting, there is a triangle equivalence between $T(A)\mbox{-\underline{\rm mod}}$ and $T(B)\mbox{-\underline{\rm mod}}$  obtained in \cite[Theorem 3.1]{Ric}. It would be of interest to compare these two triangle equivalences. If the tilting module has projective dimension at most one, these two triangle equivalences might coincide in view of \cite[Theorem 8]{TW2}.

(2) Consider the category $T(A)\mbox{-Mod}$ of arbitrary $T(A)$-modules. Using filtered colimits and \cite[Theorem 2.4]{KS}, we obtain a cotorsion pair in $T(A)\mbox{-Mod}$ and a cotorsion pair in $T(B)\mbox{-Mod}$, which still satisfy (GW1)-(GW4). Here, we have to replace ``add" by ``Add" in (GW2).  Then we obtain a triangle functor $$\mathcal{S}\colon T(A)\mbox{-\underline{\rm Mod}}\longrightarrow T(B)\mbox{-\underline{\rm Mod}},$$
 which is an equivalence  by Theorem \ref{thm:Wak} and  infinite d\'{e}vissage.

(3) We view $T(A)=A\oplus DA$ as a $\mathbb{Z}$-graded algebra with ${\rm deg} A=0$ and ${\rm deg} DA=1$. Then the category $T(A)\mbox{-gr}$ of graded $T(A)$-modules is equivalent to the module category of the repetitive algebra of $A$; in particular,  it is a Frobenius category. By \cite[Theorem II.4.9]{Hap}, there is a triangle full embedding from the bounded derived category $\mathbf{D}^b(A\mbox{-mod})$  of $A\mbox{-mod}$ to the stable category $T(A)\mbox{-\underline{gr}}$.

A graded $T(A)$-module $(X, \phi)$ consists of a graded $A$-module $X$ with a structure map $\phi\colon DA\otimes X\rightarrow X$ of degree one, which satisfies $\phi\circ (DA\otimes \phi)=0$. Then a parallel argument as in \cite[Section 1]{Wak2} carries over to graded modules, and thus we obtain a triangle equivalence
$$\mathcal{S}\colon T(A)\mbox{-\underline{gr}}\longrightarrow T(B)\mbox{-\underline{gr}}.$$
 The construction of $\mathcal{S}$ is similar to the one in \cite[Section 2]{Wak4}, where the grading shift by one appears naturally. For the details, we refer to  \cite[Section 4]{Wei}.  We might  call the above equivalence $\mathcal{S}$ a \emph{repetitive equivalence} between the algebras $A$ and $B$.  It seems that a good Wakamatsu-tilting module plays a similar role for repetitive equivalence as a tilting module for derived equivalences.

 We observe that the above repetitive equivalence $\mathcal{S}$ usually will not restrict to a derived equivalence, that is, an equivalence between $\mathbf{D}^b(A\mbox{-mod})$ and $\mathbf{D}^b(B\mbox{-mod})$; see the explicit example in \cite[Examples 3.1 and 3.2]{Wak2}, where the two algebras are not derived equivalent.
\end{rem}

\bibliography{}

\begin{thebibliography}{9999}
\bibitem{AI} {\sc I. Assem, Y. Iwanaga}, {\em Stable equivalence of representation-finite trivial extension}, J. Algebra {\bf 102} (1986), 33--38.


\bibitem{AR} {\sc M. Auslander, I. Reiten}, {\em Applications of contravariantly finite subcategories}, Adv. Math. {\bf 86} (1991), 111--152.


\bibitem{Chen} {\sc X.W. Chen}, {\em Relative singularity categories and Gorenstein-projective modules}, Math. Nachr. {\bf 284}(2-3) (2011), 199--212.

\bibitem{CPS} {\sc E. Cline, B. Parshall, L.L. Scott}, {\em Derived categories and Morita theory}, J. Algebra {\bf 104} (1986), 397--409.

\bibitem{GT} {\sc R. Gobel, J. Trlifaj}, Approximations and Endomorphism Algebras of Modules, De Gruyter Expo. Math. {\bf 41}, Walter de Gruyter, Berlin/New York, 2006.


\bibitem{Hap} {\sc D. Happel,} Triangulated Categories in the Representation Theory of Finite Dimensional Algebras, London Math. Soc., Lecture Notes Ser. {\bf 119}, Cambridge Univ. Press, Cambridge, 1988.


\bibitem{Kel} {\sc B. Keller}, {\em Derived categories and universal problems}, Comm. Algebra {\bf 19} (1991), 699--747.

\bibitem{KS} {\sc H. Krause, O. Solberg}, {\em Applications of cotorsion pairs}, J. London Math. Soc. (2) {\bf 68} (2003), 631--650.

\bibitem{MR} {\sc F. Mantese, I. Reiten}, {\em Wakamatsu tilting modules}, J. Algebra {\bf 278} (2004), 532--552.

\bibitem{Miy} {\sc Y. Miyashita}, {\em Tilting modules of finite projective dimension}, Math. Z. {\bf 193} (1986), 113--146.

\bibitem{Ric} {\sc J. Rickard}, {\em Derived equivalences and stable equivalences}, J. Pure Appl. Algebra {\bf
        61} (1989), 303--317.


\bibitem{Ta1} {\sc H. Tachikawa}, {\em Selfinjective algebras and tilting theory}, Lecture Notes Math. {\bf 1177}, 272--307, Springer-Verlag, New York/Berlin, 1986.


\bibitem{Ta} {\sc H. Tachikawa}, {\em Reflection functors and Auslander-Reiten translations for trivial extensions of hereditary algebras}, J. Algebra {\bf 90} (1984), 98--118.


\bibitem{TW2} {\sc H. Tachikawa, T. Wakamatsu}, {\em Applications of reflection functors for self-injecrive algebras}, Lecture Notes Math. {\bf 1177}, 308--327, Springer-Verlag, New York/Berlin, 1986.


\bibitem{TW1} {\sc H. Tachikawa, T. Wakamatsu}, {\em Tilting functors and stable equivalences for self-injective algebras}, J. Algebra {\bf 109} (1987), 138--165.




\bibitem{Wak3} {\sc T. Wakamatsu}, {\em Partial Coxeter functors of selfinjective algebras}, Tsukuba J. Math. {\bf 9} (1985), 171--183.

\bibitem{Wak4} {\sc T. Wakamatsu}, {\em Stable equivalence between universal covers of trivial extension self-injective algebras}, Tsukuba J. Math. {\bf 9} (1985), 299--316.



\bibitem{Wak1} {\sc T. Wakamatsu}, {\em On modules with trivial self-extensions}, J. Algebra {\bf 114} (1988), 106--114.


\bibitem{Wak2} {\sc T. Wakamatsu}, {\em Stable equivalence for self-injective algebras and a generalization of tilting modules}, J. Algebra {\bf 134} (1990), 298--325.


\bibitem{Wei} {\sc J. Wei}, {\em Repetitive equivalences and Wakamatsu-tilting modules}, arXiv:1601.0139v1.


\end{thebibliography}

\vskip 10pt

 {\footnotesize \noindent Xiao-Wu Chen\\
 Key Laboratory of Wu Wen-Tsun Mathematics, Chinese Academy of Sciences,\\
 School of Mathematical Sciences, University of Science and Technology of China, Hefei 230026, Anhui, PR China}

 \vskip 5pt

 {\footnotesize \noindent Jiaqun Wei\\
  Institute of Mathematics, School of Mathematical Sciences, Nanjing Normal University, Nanjing 210023, Jiangsu, PR China.}

\end{document}